\documentclass[11pt]{amsart}

\usepackage{amsmath,amssymb,amsthm,accents}
\usepackage{amscd}
\usepackage[initials,short-journals,non-sorted-cites]{amsrefs}
\usepackage{enumitem}
\usepackage{type1cm}
\usepackage{float}
\usepackage[all]{xy}
\usepackage{amsfonts}

\usepackage[dvipdfmx]{graphicx}
\usepackage{comment}
\usepackage{bm}

\usepackage{color}

\allowdisplaybreaks[4]

\newtheorem{theo}{Theorem}[section]
\newtheorem{prop}[theo]{Proposition}
\newtheorem{lemm}[theo]{Lemma}
\newtheorem{cor}[theo]{Corollary}

\theoremstyle{definition}

\newtheorem{remark}[theo]{Remark}

\newcommand{\Z}{\mathbb{Z}}

\newcommand{\R}{\mathbb{R}}

\title{$2$-knots with the same knot group but different knot quandles}

\author{Kokoro Tanaka}
\address{\scriptsize Department of Mathematics, Tokyo Gakugei University, 
4-1-1, Nukuikita, Koganei, Tokyo 184-8501, Japan}
\email{kotanaka@u-gakugei.ac.jp}

\author{Yuta Taniguchi}
\address{\scriptsize Department of Mathematics, Graduate School of Science, Osaka University, 
1-1, Machikaneyama, Toyonaka, Osaka, 560-0043, Japan} 
\email{yuta.taniguchi.math@gmail.com}
\date{}

\begin{document}
\renewcommand{\subjclassname}{
\textup{2020} Mathematics Subject Classification}
%
\maketitle

\begin{abstract}
We give a first example of $2$-knots with the same knot group but different knot quandles 
by analyzing the knot quandles of 
twist spins. 
As a byproduct of the analysis, we also give a classification of all 
twist spins 
with finite knot quandles. 
\end{abstract}

\section{Introduction}
A $1$-knot is a circle embedded in the $3$-sphere $S^3$, 
a $2$-knot is a $2$-sphere embedded in the $4$-sphere $S^4$, 
and they are collectively referred to as knots in this paper. 
There are two related topological invariants of an oriented knot $\mathcal{K}$ defined 
by using information of the knot complement 
with a chosen fixed base point. 
One is the \textit{knot group} $G(\mathcal{K})$, 
the fundamental group of the complement,  
defined as the set of homotopy classes of loops,  
and the other is the \textit{knot quandle} $Q(\mathcal{K})$ 
\cite{Joyce1982quandle, Matveev1982distributive}
defined as the set of homotopy classes of paths 
with a suitable binary operation. 
Note that the definition of $G(\mathcal{K})$ does not require the orientation of 
$\mathcal{K}$, but the definition of $Q(\mathcal{K})$ does.
In this paper, we investigate difference between these two invariants. 

To make the problem concrete, we consider the following two conditions 
for oriented knots $\mathcal{K}, \mathcal{K}'$: 
\begin{enumerate}
\item[(a)] 
The knot groups $G(\mathcal{K})$ and $G(\mathcal{K}')$ are isomorphic to each other. 
\item[(b)]
The knot quandles $Q(\mathcal{K})$ and $Q(\mathcal{K}')$ are isomorphic to each other. 
\end{enumerate}
Since the associated group of the knot quandle is group isomorphic to the knot group, 
it is known that the condition (b) implies the condition (a). 
The converse does not hold in general for $1$-knots. 
For example, the square knot and the granny knot satisfy the condition (a) but 
do not satisfy the condition (b), that is, 
they are $1$-knots with the same knot group but different knot quandles. 
On the other hand, there are no known examples of 2-knots with the same knot group but different knot quandles as in Remark~\ref{rem:2-knot}. 
The purpose of this paper is to give a first example for such $2$-knots.  
More precisely, we prove in Theorem~\ref{theo:main} 
that there exist infinitely many triples of $2$-knots 
such that any two of each triples satisfy the condition (a) 
but do not satisfy the condition (b). 
The main ingredient of the proof is Theorem~\ref{theo:type_twist_spun}, which reveals 
an algebraic property, called the \textit{type}, of the knot quandle of a twist-spin.

\begin{remark}\label{rem:2-knot}
Researchers in this field may have believed that there were $2$-knots with the same knot group but different knot quandles, but they would not have known of any specific examples of such $2$-knots or proofs of their existence. 
We point out that based on the example of such $1$-knots 
we cannot construct an example of such $2$-knots in a straightforward way. 
Zeeman's twist-spinning construction \cite{Zeeman1965twisting} is a typical way to obtain a $2$-knot $\tau^n(k)$ from a $1$-knot $k$ (and an integer $n$). 
For the square knot $k$ and the granny knot $k'$, it is known that $\tau^0 (k)$ and $\tau^0(k')$ are equivalent \cite{Gordon1976note, Roseman1975spun}, 
and hence they satisfy the condition (b). 
Moreover it follows that $\tau^n (k)$ and $\tau^n(k')$ 
satisfy the condition (b) for any integer $n$, 
even if we do not know whether they are equivalent or not.\footnote{
We can check that $\tau^n (k)$ and $\tau^n(k')$ are not equivalent for 
$n = 2,3$ by using quandle cocycle invariants \cite{Carter2003quandle}. 
}
\end{remark}

\begin{remark}
Based on the example of such $1$-knots,  
we can easily obtain an example of such surface knots of positive genus 
in a straightforward way, where the surface knot is a closed surface embedded in $S^4$. 
Note that a surface knot of genus zero is nothing but a $2$-knot. 
For a classical knot $k$, the spun $k$ torus, 
which is a knotted torus in $S^4$ defined by Livingston \cite{Livingston1985stably},  
has the knot quandle isomorphic to $Q(k)$ and also has the knot group isomorphic to $G(k)$. 
Then, using the above example of $1$-knots, the granny knot and the square knot, 
we obtain two surface knots of genus one with the same knot group but different knot quandles. 
We can also obtain such surface knots of genus greater than one by stabilizing them. 
\end{remark}

This paper is organized as follows. 
In Section~\ref{sec:pre}, we review the basics of quandles including 
the type of a quandle, the knot quandle of an oriented knot and a generalized Alexander quandle.  
We compute the type of the knot quandle for a twist spin (Theorem~\ref{theo:type_twist_spun}) 
and prove our main result (Theorem~\ref{theo:main}) in Section~\ref{sec:type_twist_spun}. 
As an application of Theorem~\ref{theo:type_twist_spun}, 
we give a classification of all twist spins with finite knot quandles in Section~\ref{sec:finite}. 
Appendix~\ref{appendix:proof} is devoted to proving Proposition~\ref{prop:order_automorphism}, 
which is a key proposition for the proof of Theorem~\ref{theo:type_twist_spun}. 
Finally, in Appendix~\ref{appendix:BTS}, we compute the types of the knot quandles for 
a certain class of $2$-knots called branched twist spins including all twist spins.

\section{Preliminaries}
\label{sec:pre}
A {\it quandle} $X$ \cite{Joyce1982quandle,Matveev1982distributive} is a non-empty set with a binary operation $\ast$ that satisfies the following conditions:
\begin{itemize}
  \setlength{\parskip}{0cm} 
  \setlength{\itemsep}{0cm} 
\item For any $x\in X$, we have $x\ast x=x$.
\item For any $y\in X$, the map $S_y:X\to X;x\mapsto x\ast y$ is bijective.
\item For any $x,y,z\in X$, we have $(x\ast y)\ast z=(x\ast z)\ast(y\ast z)$.
\end{itemize} 
Let $X$ be a quandle. For any $x,y\in X$ and $n\in\Z$, we denote $S_y^n(x)$ by $x\ast^n y$. We set ${\rm type}(X):={\rm min}\{ n\in\Z_{>0}\mid x\ast^n y=x\ (\textrm{for any }x,y\in X)\}$, where ${\rm min}\ \emptyset=\infty$. We call ${\rm type}(X)$ the {\it type} of the quandle $X$.

The {\it associated group} of $X$ is the group generated by the all elements of $X$ subject to the relations $x\ast y=y^{-1}xy$ for all $x,y\in X$. We denote the associated group of $X$ by ${\rm As}(X)$.  The associated group of $X$ acts on $X$ from the right by $x\cdot y:=x\ast y$ for any $x,y\in X$. A quandle $X$ is {\it connected} if the action of ${\rm As}(X)$ on $X$ is transitive.

A map $f:X\to Y$ between quandles is a {\it $($quandle$)$ isomorphism} if $f(x\ast y)=f(x)\ast f(y)$ for any $x,y\in X$ and $f$ is a bijection. When there is a quandle isomorphism $f:X\to Y$, we say that $X$ and $Y$ are {\it $($quandle$)$ isomorphic}.

Let $\mathcal{K}$ be an oriented $n$-knot for $n=1,2$. Let $N(\mathcal{K})$ be a tubular neighborhood of $\mathcal{K}$ and $E(\mathcal{K})=S^{n+2}\backslash {\rm int}N(\mathcal{K})$ an exterior of $\mathcal{K}$. We fix a point $p\in E(\mathcal{K})$. Let $Q(\mathcal{K},p)$ be the set of homotopy classes of all paths in $E(\mathcal{K})$ from a point in $\partial E(\mathcal{K})$ to $p$. The set $Q(\mathcal{K},p)$ is a quandle with an operation defined by $[\alpha]\ast[\beta]:=[\alpha\cdot\beta^{-1}\cdot m_{\beta(0)}\cdot \beta]$, where $m_{\beta(0)}$ is a meridian loop starting from $\beta(0)$ and going along in the positive direction. We call $Q(\mathcal{K},p)$ the {\it knot quandle} of $\mathcal{K}$, 
which is known to be connected. 
The isomorphism class of the knot quandle does not depend on the base point $p$. Thus, we denote the knot quandle simply by $Q(\mathcal{K})$. We note that for an oriented knot $\mathcal{K}$, the associated group ${\rm As}(Q(\mathcal{K}))$ is group isomorphic to the knot group $G(\mathcal{K}):=\pi_1(E(\mathcal{K}))$.


Let $G$ be a group and $f:G\to G$ a group automorphism. We define the operation $\ast$ on $G$ by $x\ast y:=f(xy^{-1})y$. Then, ${\rm GAlex}(G,f)=(G,\ast)$ is a quandle, which is called the {\it generalized Alexander quandle}. 
Although it is difficult to compute the type of a quandle in general, 
the type of a generalized Alexander quandle can be computed as follows,  
which will be used in what follows.

\begin{prop}
\label{prop:type_generalized}
We have ${\rm type}({\rm GAlex}(G,f)) = {\rm order}(f)$. 
\end{prop}
\begin{proof}
Since the definition of ${\rm GAlex}(G,f)$, we have 
\[
(f^i(xy^{-1})y)\ast y=f((f^i(xy^{-1})y)y^{-1})y=f(f^i(xy^{-1}))y=f^{i+1}(xy^{-1})y
\]
 for any $x,y\in G$ and $i\in \Z_{\ge 0}$. Thus, we see that $x\ast^iy=f^i(xy^{-1})y$ for any $x,y\in G$ and $i\in\Z_{\ge 0}$. 
 
Let $j$ be a positive integer. For any $x\in G$, $x\ast^j e$ is equal to $f^j(x)$, where $e$ is the identity element of $G$. Thus, if $j<{\rm order}(f)$, there exists an element $y$ such that $y\ast^j e$ is not equal to $y$. This implies that ${\rm type}({\rm GAlex}(G,f))\geq{\rm order}(f)$. In particular, if ${\rm order}(f)=\infty$, then we have ${\rm GAlex}(G,f)=\infty$. Hence, we assume that the order of $f$ is finite.
 
 For any $x,y\in G$, we have 
 \[
 x\ast^{{\rm order}(f)}y=f^{{\rm order}(f)}(xy^{-1})y=xy^{-1}y=x,
 \]
  which implies that ${\rm type}({\rm GAlex}(G,f))\leq{\rm order}(f)$. By the above discussion, we see that ${\rm order}(f)={\rm type}({\rm GAlex}(G,f))$.
\end{proof}

\section{Knot quandles of twist spins}
\label{sec:type_twist_spun}
In this section, we discuss the knot quandle of twist spins. 
The purpose of this section is to show the following theorem:
\begin{theo}
\label{theo:main}
There exists a triple $\{ F_1,F_2,F_3\}$ of 2-knots such that
\begin{itemize}
\item[$(1)$] the groups 
$G(F_1),G(F_2)$ and $G(F_3)$ are mutually isomorphic, but 
\item[$(2)$] no two of the quandles 
$Q(F_1),Q(F_2)$ and $Q(F_2)$ are isomorphic.
\end{itemize} 
Moreover, there exist infinitely many such triples. 
\end{theo}

Let $k$ be an oriented 1-knot in $S^3$. For an integer $n$, Zeeman \cite{Zeeman1965twisting} defined the $2$-knot $\tau^n(k)$, which is called the {\it $n$-twist spun $k$}. 
It was shown 
that the $n$-twist spun $k$ is a fibered 2-knot whose fiber is the once punctured $M^{n}_k$ for $n > 0$, where $M^{n}_k$ is the $n$-fold cyclic branched covering space of $S^3$ branched along $k$. 
In particular, the $1$-twist spun $k$ is trivial for any oriented 1-knot $k$, 
and the $n$-twist spin of the trivial $1$-knot is also trivial for any $n>0$. 
Litherland \cite{Litherland1985symmetries} showed that, for any oriented 1-knot $k$ and any integer $n$, the $(-n)$-twist spun knot $\tau^{-n}(k)$ is equivalent to $\tau^n(-k!)$, 
where $-k$ (resp.~$k!$) is the oriented $1$-knot obtained from $k$ 
by inverting the orientation (resp.~taking the mirror image) of $k$. 
Hence, we consider the case where $n>1$ and $k$ is non-trivial, 
since we are interested in fibered $2$-knots. 

Let $\varphi$ be the group automorphism of $\pi_1(M^n_k)$ induced by the monodromy of $S^4\backslash \tau^n(k)$. We note that the monodromy of $S^4\backslash \tau^n(k)$ coincides with the canonical generator of the transformation group of $M^n_k$. 
Then we have the following proposition, whose proof will be given in Appendix \ref{appendix:proof}: 

\begin{prop}
\label{prop:order_automorphism}
Let $\varphi$ be the group automorphism of $\pi_1(M^n_k)$ induced by the monodromy of $S^4\backslash \tau^n(k)$. Then the order of $\varphi$ is equal to $n$.
\end{prop}

Inoue \cite[Theorem~$3.1$]{Inoue2019fibered} showed that the knot quandle $Q(\tau^n(k))$ of $\tau^n(k)$ is quandle isomorphic to ${\rm GAlex}(\pi_1(M^n_k),\varphi)$. 
Then, we have: 

\begin{theo}
\label{theo:type_twist_spun}
The type of $Q(\tau^n(k))$ is equal to $n$.
\end{theo}
\begin{proof}
It follows from Proposition~\ref{prop:type_generalized} and Proposition \ref{prop:order_automorphism} 
that we have ${\rm type}(Q(\tau^n(k))) = {\rm type}\left( {\rm GAlex}(\pi_1(M^n_k),\varphi) \right) 
= {\rm order}( \varphi ) = n$.
\end{proof}
\begin{cor}
\label{cor:type_twist_spun}
Let $k, k^{\prime}$ be non-trivial oriented $1$-knots and $n,n^{\prime}$ integers greater than $1$. 
If $Q(\tau^n(k))$ is quandle isomorphic to $Q(\tau^{n^\prime}(k^\prime))$, then we have $n=n^\prime$.
\end{cor}
\begin{proof}
Since the type of quandles is an invariant of quandles, it holds that 
$$
n = {\rm type}(Q(\tau^n(k))) = {\rm type}(Q(\tau^{n^{\prime}}(k^{\prime})))=n^{\prime}
$$
by Theorem \ref{theo:type_twist_spun} and the assumption. 
\end{proof}

\begin{proof}[Proof of Theorem \ref{theo:main}]
Let $p,q$ and $r$ be coprime integers. Gordon \cite{Gordon1972twist} showed that $G(\tau^r(t_{p,q}))$ is isomorphic to $\pi_1(M^r_{t_{p,q}})\times\Z$, where $t_{p,q}$ is the $(p,q)$-torus knot. It is known that $M^p_{t_{q,r}},M^q_{t_{r,p}}$ and $M^r_{t_{p,q}}$ are homeomorphic (\cite{Milnor1975Brieskorn}), which implies that $G(\tau^p(t_{q,r})),G(\tau^q(t_{r,p}))$ and $G(\tau^r(t_{p,q}))$ are mutually group isomorphic. Thus, putting $F_1:=\tau^p(t_{q,r}),F_2:=\tau^q(t_{r,p})$ and $F_3:=\tau^r(t_{p,q})$, we see that 2-knots $F_1,F_2$ and $F_3$ satisfy the condition (1). 
By Theorem \ref{theo:type_twist_spun}, we have ${\rm type}(Q(\tau^p(t_{q,r})))=p,\ {\rm type}(Q(\tau^q(t_{r,p})))=q$ and ${\rm type}(Q(\tau^r(t_{p,q})))=r$. This implies that the 2-knots $F_1,F_2$ and $F_3$ satisfy the condition (2). 
Varying a triple of comprime integers, we obtain infinitely many such triples of $2$-knots.  
\end{proof}

\section{Classification of twist spins whose knot quandles are finite}
\label{sec:finite}
In this section, we give a classification of twist spins whose knot quandles are finite. 
%
%
Let $k$ be an oriented $1$-knot and $n$ an integer greater than $1$. Since the knot quandle $Q(\tau^n (k))$ is isomorphic to the generalized Alexander quandle ${\rm GAlex}(\pi_1(M^n_k),\varphi)$, the knot quandle $Q(\tau^n (k))$ is finite if and only if the fundamental group $\pi_1(M^n_k)$ is finite. Thus, if the knot quandle $Q(\tau^n(k))$ is finite, $\tau^n (k)$ is 
an element of one of the following six sets (see \cite{Inoue2019fibered}):  
{\setlength{\leftmargini}{7mm} 
\begin{itemize}
\item $S_1=\{\tau^n (k)\mid n=2,k \textrm{: a non-trivial 2-bridge knot}\}$.
\item $S_2=\{\tau^n (k)\mid n=2,k \textrm{: a Montesinos knot }M((2,\beta_1),(3,\beta_2),(3,\beta_3))\}$.
\item $S_3=\{\tau^n (k)\mid n=2,k \textrm{: a Montesinos knot }M((2,\beta_1),(3,\beta_2),(5,\beta_3))\}$.
\item $S_4=\{\tau^n (k)\mid n=3,k \textrm{: the torus knot $t_{2,3}$ or the torus knot $t_{2,5}$}\}$.
\item $S_5=\{\tau^n (k)\mid n=4,k \textrm{: the torus knot $t_{2,3}$}\}$.
\item $S_6=\{\tau^n (k)\mid n=5,k \textrm{: the torus knot $t_{2,3}$}\}$.
\end{itemize}
}

\subsection{Classification as unoriented $2$-knots}
\label{subsec:finite-unori}
We classify twist spins whose knot quandles are finite as unoriented $2$-knots. 

\begin{theo}\label{theo:finite-unori}
Let $k$, $k^{\prime}$ be non-trivial oriented $1$-knots and $n$, $n^{\prime}$ integers greater than $1$. 
Suppose that $Q(\tau^n (k))$ and $Q(\tau^{n^{\prime}}(k^{\prime}))$ are finite. 
Then $\tau^n(k)$ is equivalent to $\tau^{n^{\prime}}(k^{\prime})$ as unoriented $2$-knots 
if and only if 
$n=n^{\prime}$ and $k$ is equivalent to $k^{\prime}$ as unoriented $1$-knots.
\end{theo}
\begin{proof}
It is obvious that $\tau^n(k)$ and $\tau^{n^{\prime}}(k^{\prime})$ are equivalent if $n=n^{\prime}$ and $k$ and $k^{\prime}$ are equivalent. We discuss the converse. 
Suppose that $\tau^n(k)$ and $\tau^{n^{\prime}}(k^{\prime})$, 
whose knot quandles are finite, are equivalent. 

\medskip
\underline{Case 1.} 
Consider the case where $\tau^n(k)$ is an element of $S_1\cup S_2\cup S_3$.  
By Corollary \ref{cor:type_twist_spun}, we have $n^{\prime}=n=2$. Since $Q(\tau^{n^{\prime}}(k^{\prime}))$ is finite, $\tau^{n^{\prime}}(k^{\prime})$ is also an element of $S_1\cup S_2\cup S_3$. By \cite{Jang2020twist}, we see that $k$ and $k^{\prime}$ are equivalent.

\medskip
\underline{Case 2.} 
Consider the case where $\tau^n(k)$ is an element of $S_4$.   
By Corollary \ref{cor:type_twist_spun}, we have $n^{\prime}=n=3$. Since $Q(\tau^{n^{\prime}}(k^{\prime}))$ is finite, $\tau^{n^{\prime}}(k^{\prime})$ is also an element of $S_4$, and hence $k$ and $k^{\prime}$ are equivalent to $t_{2,3}$ or $t_{2,5}$. It is known that 
\[
|Q(\tau^3 (t_{2,3}))|=|\pi_1(M^3_{t_{2,3}})|=8  \ \  \text{and} \ \  
|Q(\tau^3( t_{2,5}))|=|\pi_1(M^3_{t_{2,5}})|=120,
\] 
which implies that $\tau^3(t_{2,3})$ and $\tau^3(t_{2,5})$ are not equivalent. 
Thus, we see that $k$ and $k^{\prime}$ are equivalent.

\medskip
\underline{Case 3.} 
Consider the case where $\tau^n(k)$ is an element of $S_5\cup S_6$. 
By Corollary \ref{cor:type_twist_spun}, we have $n=n^{\prime} \in \{4,5\}$. 
Since $Q(\tau^{n^{\prime}}(k^{\prime}))$ is finite, $\tau^{n^{\prime}}(k^{\prime})$ is also an element of $S_5\cup S_6$, and hence $k$ and $k^{\prime}$ are equivalent to $t_{2,3}$. 
\end{proof}

\begin{remark}
Based on the classification of $\tau^2(t_{3,5})$, $\tau^3(t_{5,2})$ and $\tau^5(t_{2,3})$ 
by Gordon \cite[Proof of Theorem $1.1$]{Gordon1972twist} and  
that of the set $S_1$, which consists of all $2$-twist spun $2$-bridge knots,  
by Plotnick~\cite{Plotnick1983homotopy}, 
our Theorem~\ref{theo:finite-unori} has almost been proved in a series of the works by  
Kataoka~\cite{Kataoka2020twist}, Jang--Kataoka--Miyakoshi~\cite{Jang2020twist} and Miyakoshi~\cite{Miyakoshi2022twist}. 
Details are as follows.

Kataoka~\cite{Kataoka2020twist} classified the set $S_1 \cup S_2 \cup S_3$, 
which consists of all $2$-twist spun spherical Montesinos knots,  
by computing the orders of the knot quandles, 
and using dihedral group representations of the knot groups 
(and the classification by Plotnick~\cite{Plotnick1983homotopy} mentioned above). 
Then Jang--Kataoka--Miyakoshi~\cite{Jang2020twist} simplified its proof 
by using $3$-colorings rather than dihedral group representations. 
With these in mind, Miyakoshi~\cite{Miyakoshi2022twist} almost proved our Theorem~\ref{theo:finite-unori} by further computing the orders of the knot quandles for the set $S_4 \cup S_5 \cup S_6$,  
and using $3$-colorings
(and the classification by Gordon \cite{Gordon1972twist} mentioned above). 
The only remaining part was whether $\tau^2(t_{3,4}) \in S_2$ is equivalent to 
$\tau^4(t_{2,3}) \in S_5$ or not, 
where the torus knot $t_{3,4}$ is nothing but the Montesinos knot $M((2,-1),(3,1),(3,1))$. 
Note that, for both $2$-knots, 
the order of the knot quandle is $24$ and the number of $3$-colorings is $9$.  
\end{remark}

\subsection{Classification as oriented $2$-knots}
We classify twist spins whose knot quandles are finite as oriented $2$-knots. 
For an oriented knot $\mathcal{K}$, 
we denote by $-\mathcal{K}$ (or $\mathcal{K}!$, respectively) the oriented knot obtained from an oriented knot $\mathcal{K}$ by inverting the orientation (or taking the mirror image) of $\mathcal{K}$. 
An oriented knot $\mathcal{K}$ is said to be {\it invertible} (resp.~{\it $(+)$-amphicheiral}) if $\mathcal{K}$ is equivalent to $-\mathcal{K}$ (resp.~$\mathcal{K}!$). 
Litherland \cite{Litherland1985symmetries} showed that, for an oriented $1$-knot $k$ and an integer $n$, 
$\tau^n(-k)$ is equivalent to $(\tau^n(k))!$ and $\tau^n(k!)$ is equivalent to $-(\tau^n (k))$ as oriented $2$-knots. 
Hence, if the oriented $1$-knot $k$ is invertible (resp.~$(+)$-amphicheiral), 
then the oriented $2$-knot $\tau^n(k)$ is $(+)$-amphicheiral (resp.~invertible). 
Although it is not known whether or not the converse holds in general, 
we can say the following for twist spins whose knot quandles are finite:

\begin{prop}\label{prop:finite-ori}
Suppose that $Q(\tau^n (k))$ is finite for a non-trivial oriented $1$-knot $k$ and 
an integer $n$ greater than $1$. 

\begin{enumerate}
\item[$(1)$]
$k$ is invertible and $\tau^n(k)$ is $(+)$-amphicheiral. 
\item[$(2)$]
$k$ is $(+)$-amphicheiral if and only if $\tau^n(k)$ is invertible. 
\end{enumerate}
\end{prop}
\begin{proof}
We prove the case (1). 
Since $Q(\tau^n (k))$ is finite, it follows from Subsection~\ref{subsec:finite-unori} that 
the $1$-knot $k$ is a 2-bridge knot or a Montesinos knot of 
type $((2,\beta_1),(3,\beta_2),(\alpha_3,\beta_3))$, where $\alpha_3=3$ or $5$, 
which implies that $1$-knot $k$ is invertible (see \cite{Burde2014knots}). 
Hence, by \cite{Litherland1985symmetries}, the $2$-knot $\tau^n(k)$ is $(+)$-amphicheiral. 

We prove the case (2). 
Since $Q(\tau^n (k))$ is finite, the $2$-knot $\tau^n(k)$ is an element of $S_1 \cup \cdots \cup S_6$. 
Gordon \cite{Gordon2003reversibility} showed that, 
when the $2$-knot $\tau^n(k)$ is an element of $S_1$, 
the $1$-knot $k$ is $(+)$-amphicheiral if and only if the $2$-knot $\tau^n(k)$ is invertible. 
Gordon \cite{Gordon2003reversibility} also showed that 
the $n$-twist spin of any torus knot is not invertible for $n>1$. 
Since it is known that any torus knot is not $(+)$-amphicheiral, 
when the $2$-knot $\tau^n(k)$ is an element of $S_4 \cup S_5 \cup S_6$, 
the $1$-knot $k$ is not $(+)$-amphicheiral and the $2$-knot $\tau^n(k)$ is not invertible. 
Hence, it is sufficient to consider the case where the $2$-knot $\tau^n(k)$ is an element in $S_2\cup S_3$. 
In this case, we have that $n=2$ and the $1$-knot $k$ is a Montesinos knot of type $((2,\beta_1),(3,\beta_2),(\alpha_3,\beta_3))$, where $\alpha_3=3$ or $5$. 
Since it is known that the $1$-knot $k$ is not $(+)$-amphicheiral, 
all we need is to prove that the $2$-knot $\tau^2 (k)$ is not invertible. 
To prove it, we recall the following proposition due to Gordon \cite{Gordon2003reversibility}: 
\begin{prop}
\label{prop:fiber_ori_reversing}
Let $M$ be a closed,connected, orientable 3-manifold and $F$ a fibered 2-knot whose fiber is the once punctured $M$. If $F$ is invertible, then there exists an orientation reversing self homotopy equivalent of $M$.
\end{prop}
The fiber $M^2_k$ of the $2$-knot $\tau^2 (k)$ is known to have the finite non-abelian fundamental group. 
It follows from the proof of Theorem 8.2 of \cite{Neumann1978Seifert} that 
the fiber $M^2_k$ does not have an orientation reversing self homotopy equivalent. 
Thus, by Proposition~\ref{prop:fiber_ori_reversing}, we see that the $2$-knot $\tau^2(k)$ is not invertible.
\end{proof}

Combining Theorem~\ref{theo:finite-unori} with Proposition~\ref{prop:finite-ori}, we have the following: 

\begin{theo}
Let $k$, $k^{\prime}$ be non-trivial oriented $1$-knots and $n$, $n^{\prime}$ integers greater than $1$. Suppose that $Q(\tau^n (k))$ and $Q(\tau^{n^{\prime}}(k^{\prime}))$ are finite. Then $\tau^n(k)$ is equivalent to $\tau^{n^{\prime}}(k^{\prime})$ as oriented $2$-knot if and only if $n=n^{\prime}$ and $k$ is equivalent to $k^{\prime}$ as oriented $1$-knot.
\end{theo}

\appendix
\section{Proof of Proposition \ref{prop:order_automorphism}}
\label{appendix:proof}
In this appendix, we will show Proposition \ref{prop:order_automorphism}. 
For a non-trivial oriented $1$-knot $k$ in $S^3$ and an integer $n$ greater than 1, 
let $M^n_k$ be the $n$-fold cyclic branched covering space of $S^3$ branched along $k$, 
and $\varphi$ the group automorphism of $\pi_1(M^n_k)$ induced 
by the monodromy of the complement $S^4\backslash\tau^n(k)$ of the fibered $2$-knot $\tau^n(k)$. 
Then, by~\cite[Proposition A.11.]{Conner1972manifolds} and~\cite[Lemma $2.3$]{Plotnick1983homotopy}, 
we have the following: 

\begin{lemm}
\label{lemm:order_group_auto}
If the universal covering space of $M^n_k$ is homeomorphic to $\R^3$, 
then the order of $\varphi$ is equals to $n$. 
\end{lemm}

Using the lemma above, we give a proof of Proposition \ref{prop:order_automorphism}, 
in which we call a $1$-knot in $S^3$ as a knot for simplicity. 

\begin{proof}[Proof of Proposition \ref{prop:order_automorphism}]
If ${\rm order}(\varphi)=1$, the quandle ${\rm GAlex}(\pi_1(M^n_k),\varphi)$ is trivial, that is, it holds that $x\ast y=x$ for any $x,y\in \pi_1(M^n_k)$. 
Since ${\rm GAlex}(\pi_1(M^n_k),\varphi)$ is isomorphic to the knot quandle $Q(\tau^n(k))$, the quandle ${\rm GAlex}(\pi_1(M^n_k),\varphi)$ is connected. This implies that the cardinality of ${\rm GAlex}(\pi_1(M^n_k),\varphi)$ is $1$, that is, 
$\pi_1(M^n_k)$ is trivial. 
By the Smith conjecture, the knot $k$ must be the unknot. Hence, we have ${\rm order}(\varphi) \neq 1$. 

Since $\varphi$ is induced by the canonical generator of the transformation group of $M^n_k$, we see that $\varphi^n$ is the identity map, and hence $n$ is divisible by ${\rm order}(\varphi)$. 
Then, if $n$ is a prime number, it follows from ${\rm order}(\varphi) \neq 1$ that 
${\rm order}(\varphi)=n$. Hence, we assume that $n$ is a composite number.

If $k$ is the composite knot of $k_1$ and $k_2$, then the fundamental group $\pi_1(M^n_k)$ is the free product of $\pi_1(M^n_{k_1})$ and $\pi_1(M^n_{k_2})$ (see \cite[Theorem 2]{Sakuma1982regular}). For $i=1,2$, let $\varphi_i$ be the group automorphism induced by the canonical generator of the transformation group of $M^n_{k_i}$, and $\eta_i$ the injective group homomorphism from $\pi_1(M^n_{k_i})$ to $\pi_1(M^n_k)$. It holds that the restriction $\varphi|_{\pi_1(M^n_{k_i})}$ coincides with the group homomorphism $\eta_i\circ\varphi_i:\pi_1(M^n_{k_i})\to \pi_1(M^n_k)$ for $i\in\{1,2\}$. Thus, if ${\rm order}(\varphi_1)={\rm order}(\varphi_2)=n$, we see that ${\rm order}(\varphi)=n$. 
Hence it is sufficient to consider the case where $k$ is a prime knot, 
that is, $k$ is one of a torus knot, a hyperbolic knot or a satellite knot. 
We note that, by the equivalent sphere theorem \cite{Meeks1980topology}, 
the branched covering space $M^n_k$ of the prime knot $k$ 
is irreducible (see also \cite[Theorem 2]{Sakuma1982regular}), 
which will be used in Case $1$ and Case $3$ below.  

\medskip
\underline{Case 1.} Consider the case where $k$ is a torus knot. 
Since $S^3\backslash k$ is a Seifert fibered space, 
$M^n_k$ is also a Seifert fibered space. 
Thus, the universal covering space of $M^n_k$ is either $S^3,S^2\times\R$ or $\R^3$ 
(see \cite[Lemma 3.1]{Scott1983geometries}). 
{\setlength{\leftmargini}{5mm}  
\begin{itemize}
\item Suppose that the universal covering space of $M^n_k$ is $S^2\times\R$. Then $\pi_2(M^n_k)$ is non-trivial. By the sphere theorem \cite{Papakyriakopoulos1957Dehn}, it holds that $M^n_k$ is not irreducible. This is a contradiction.
\item  Suppose that universal covering space of $M^n_k$ is $S^3$. 
Since $n$ is a composite number, we see that $k = t_{2,3}$ and $n=4$. 
Then we have ${\rm order}(\varphi) \mid 4$. 
Moreover, it follows from ${\rm order}(\varphi) \neq 1$ that ${\rm order}(\varphi) = 2$ or $4$. 
It is known that $Q( \tau^4(t_{2,3}) )$ 
has the following presentation (cf. \cite{Satoh2002surface}):
\[
\langle a,b\mid (a\ast b)\ast a=b, a\ast^4 b=a\rangle. 
\]
If ${\rm order}(\varphi)=2$, then ${\rm type}(Q(\tau^4( t_{2,3})))=2$ 
by Proposition~\ref{prop:type_generalized}. 
Hence, the following is also a presentation of $Q(\tau^4 t_{2,3})$:
\[
\langle a,b\mid (a\ast b)\ast a=b, a\ast^4 b=a, 
a\ast^2 b=a, b\ast^2 a=b \rangle.
\]
Using Tietze's moves, we see that $Q(\tau^4 (t_{2,3}))$ has the following presentation:
\[
\langle a,b\mid (a\ast b)\ast a=b, a\ast^2 b=a\rangle.
\]
This implies that $Q(\tau^4 (t_{2,3}))$ is quandle isomorphic to $Q(\tau^2 (t_{2,3}))$. On the other hand, by Inoue's result, we have 
\[
|Q(\tau^4 (t_{2,3}))|=|\pi_1(M^4_{t_{2,3}})|=24 \ \ \text{and} \ \ 
|Q(\tau^2 (t_{2,3}))|=|\pi_1(M^2_{t_{2,3}})|=3. 
\]
This is a contradiction. 
Hence we have ${\rm order}(\varphi) = 4$. 

\item Suppose that the universal covering space of $M^n_k$ is $\R^3$. 
By Lemma \ref{lemm:order_group_auto}, we have ${\rm order}(\varphi)=n$.
\end{itemize}
}	

\medskip
\underline{Case 2.} Consider the case where $k$ is a hyperbolic knot. 
Since $n$ is a composite number, $n$ is greater than $3$. 
As a consequence of the orbifold theorem, it holds that $M^n_k$ is a hyperbolic manifold (see \cite{Boileau2001geometrization,Cooper2000Three}). Thus, the universal covering space of $M^n_k$ is $\R^3$. 
By Lemma \ref{lemm:order_group_auto}, we have ${\rm order}(\varphi)=n$.

\medskip
\underline{Case 3.} Consider the case where $k$ is a satellite knot. 
In this case, the complement $E(k)$ of $k$ contains an incompressible torus. 
Then, by \cite{Gordon1984incompressible}, we have 
either that 
$M^n_k$ is sufficiently large, 
or that $S^3$ and $M^n_k$ contain a non-separating 2-sphere. 
Since $S^3$ does not contain non-separating 2-sphere, we see that $M^n_k$ is sufficiently large. 
Recall that 
$M^n_k$ is also irreducible as metioned above. 
Since it is shown in  \cite{Waldhausen1968irreducible} that 
the universal covering space of the interior of a sufficiently large irreducible 3-manifold is $\R^3$, 
the universal covering space of $M^n_k$ is $\R^3$. 
By Lemma \ref{lemm:order_group_auto}, we have ${\rm order}(\varphi)=n$.
\end{proof}

\section{Branched twist spin}
\label{appendix:BTS}

In this appendix, we compute the types of the knot quandles for a certain class of $2$-knots called branched twist spins including all twist spins. 
We also give an alternative proof of \cite[Theorem 1.1]{FukudaRepresentations} quoted below as Theorem~\ref{theo:Fukuda}, 
which followed from \cite[Theorem 4.1]{FukudaRepresentations} on knot group representations for branched twist spins. 

A $2$-knot that is invariant under a circle action on the $4$-sphere $S^4$ 
is called a \textit{branched twist spin}. 
Let $\tau^{n,s}(k)$ be the branched twist spin obtained from an oriented $1$-knot $k$ and 
a pair of coprime positive integers $n$ and $s$ with $n>1$. 
Since the branched twist spin $\tau^{n,s}(k)$ is 
the branch set of 
the $s$-fold cyclic branched covering space of $S^4$ branched along the twist spin $\tau^n(k)$, 
it is a fibered $2$-knot whose fiber is the once punctured 
$n$-fold cyclic branched covering space of $S^3$ branched along $k$ 
and whose monodromy is the $s$-times composite of the canonical generator of the transformation group of $M_k^n$. 
Note that $\tau^{n,1}(k)$ is nothing but $\tau^n(k)$. 
By \cite[Theorem~$3.1$]{Inoue2019fibered},  
the knot quandle $Q(\tau^{n,s}(k))$ is quandle isomorphic to the generalized Alexander quandle 
${\rm GAlex}(\pi_1(M^n_k),\varphi^s)$, where $\varphi$ is the group automorphism of $\pi_1(M^n_k)$ 
as before. Then, for a non-trivial oriented $1$-knot $k$, we have the following:   

\begin{theo}
\label{theo:type_BTS}
The type of $Q(\tau^{n,s}(k))$ is equal to $n$.
\end{theo}
\begin{proof}
It follows from Proposition \ref{prop:order_automorphism} and the coprimeness of $n$ and $s$ 
that ${\rm order}(\varphi^s)= {\rm order}(\varphi)=n$. 
Hence, by Proposition~\ref{prop:type_generalized}, the type of $Q(\tau^{n,s}(k))$ is equal to $n$.
\end{proof}

Fukuda \cite{FukudaRepresentations} studied dihedral group representations of the knot group 
and showed the following as its application:

\begin{theo}[{\cite[Theorem 1.1]{FukudaRepresentations}}]
\label{theo:Fukuda}
Let $k_1$ and $k_2$ be non-trivial oriented $1$-knots, and $\tau^{n_1,s_1}(k_1)$ and $\tau^{n_2,s_2}(k_2)$ be branched twist spins. If $n_1$ and $n_2$ are different, then $\tau^{n_1,s_1}(k_1)$ and $\tau^{n_2,s_2}(k_2)$ are not equivalent.
\end{theo}

Here we give an alternative proof of his theorem 
by using  the knot quandle rather than the knot group. 

\begin{proof}[Proof of Theorem~$\ref{theo:Fukuda}$]
We prove the contraposition. 
If two branched twist spins $\tau^{n_1,s_1}(k_1)$ and $\tau^{n_2,s_2}(k_2)$ are equivalent, then 
the two knot quandles $Q(\tau^{n_1,s_1}(k_1))$ and $Q(\tau^{n_2,s_2}(k_2))$ are quandle isomorphic, and 
hence we have 
$$n_1 = \mathop{\textrm{type}}( Q(\tau^{n_1,s_1}(k_1)) ) =  
\mathop{\textrm{type}}( Q(\tau^{n_2,s_2}(k_2)) ) = n_2 $$
by Theorem~\ref{theo:type_BTS}. 
\end{proof}

\section*{Acknowledgement}
The authors would like to thank 
Seiichi Kamada and Makoto Sakuma 
for their helpful comments. 
The first author was supported by JSPS KAKENHI Grant Number 17K05242 and 21K03220. 
The second author was supported by JSPS KAKENHI Grant Number 21J21482.  
 
\bibliographystyle{plain}
\bibliography{reference}

\begin{bibdiv}
\begin{biblist}

\bib{Boileau2001geometrization}{article}{
      author={Boileau, Michel},
      author={Porti, Joan},
       title={Geometrization of 3-orbifolds of cyclic type},
        date={2001},
        ISSN={0303-1179},
     journal={Ast\'{e}risque},
      number={272},
       pages={208},
        note={Appendix A by Michael Heusener and Porti},
      review={\MR{1844891}},
}

\bib{Burde2014knots}{book}{
      author={Burde, Gerhard},
      author={Zieschang, Heiner},
      author={Heusener, Michael},
       title={Knots},
     edition={extended},
      series={De Gruyter Studies in Mathematics},
   publisher={De Gruyter, Berlin},
        date={2014},
      volume={5},
        ISBN={978-3-11-027074-7; 978-3-11-027078-5},
      review={\MR{3156509}},
}

\bib{Carter2003quandle}{article}{
      author={Carter, J~Scott},
      author={Jelsovsky, Daniel},
      author={Kamada, Seiichi},
      author={Langford, Laurel},
      author={Saito, Masahico},
       title={Quandle cohomology and state-sum invariants of knotted curves and
  surfaces},
        date={2003},
     journal={Trans. Amer. Math. Soc.},
      volume={355},
      number={10},
       pages={3947\ndash 3989},
}

\bib{Conner1972manifolds}{inproceedings}{
      author={Conner, P.~E.},
      author={Raymond, Frank},
       title={Manifolds with few periodic homeomorphisms},
        date={1972},
   booktitle={Proceedings of the {S}econd {C}onference on {C}ompact
  {T}ransformation {G}roups ({U}niv. {M}assachusetts, {A}mherst, {M}ass.,
  1971), {P}art {II}},
      series={Lecture Notes in Math., Vol. 299},
   publisher={Springer, Berlin},
       pages={1\ndash 75},
      review={\MR{0358835}},
}

\bib{Cooper2000Three}{book}{
      author={Cooper, Daryl},
      author={Hodgson, Craig~D.},
      author={Kerckhoff, Steven~P.},
       title={Three-dimensional orbifolds and cone-manifolds},
      series={MSJ Memoirs},
   publisher={Mathematical Society of Japan, Tokyo},
        date={2000},
      volume={5},
        ISBN={4-931469-05-1},
        note={With a postface by Sadayoshi Kojima},
      review={\MR{1778789}},
}

\bib{FukudaRepresentations}{article}{
      author={Fukuda, Mizuki},
       title={Representations of branched twist spins with a non-trivial center
  of order 2},
        note={aveilable at arXiv:2209.11583},
}

\bib{Gordon1972twist}{article}{
      author={Gordon, C.~McA.},
       title={Twist-spun torus knots},
        date={1972},
        ISSN={0002-9939},
     journal={Proc. Amer. Math. Soc.},
      volume={32},
       pages={319\ndash 322},
         url={https://doi.org/10.2307/2038356},
}

\bib{Gordon1976note}{article}{
      author={Gordon, C.~McA.},
       title={A note on spun knots},
        date={1976},
        ISSN={0002-9939},
     journal={Proc. Amer. Math. Soc.},
      volume={58},
       pages={361\ndash 362},
         url={https://doi.org/10.2307/2041417},
      review={\MR{413119}},
}

\bib{Gordon1984incompressible}{incollection}{
      author={Gordon, C.~McA.},
      author={Litherland, R.~A.},
       title={Incompressible surfaces in branched coverings},
        date={1984},
   booktitle={The {S}mith conjecture ({N}ew {Y}ork, 1979)},
      series={Pure Appl. Math.},
      volume={112},
   publisher={Academic Press, Orlando, FL},
       pages={139\ndash 152},
         url={https://doi.org/10.1016/S0079-8169(08)61639-6},
      review={\MR{758466}},
}

\bib{Gordon2003reversibility}{article}{
      author={Gordon, Cameron~McA.},
       title={On the reversibility of twist-spun knots},
        date={2003},
        ISSN={0218-2165},
     journal={J. Knot Theory Ramifications},
      volume={12},
      number={7},
       pages={893\ndash 897},
         url={https://doi.org/10.1142/S0218216503002822},
      review={\MR{2017959}},
}

\bib{Inoue2019fibered}{article}{
      author={Inoue, Ayumu},
       title={On the knot quandle of a fibered knot, finiteness and equivalence
  of knot quandles},
        date={2019},
        ISSN={0166-8641},
     journal={Topology Appl.},
      volume={265},
       pages={106811, 8},
         url={https://doi.org/10.1016/j.topol.2019.07.005},
}

\bib{Jang2020twist}{article}{
      author={Jang, Yeonhee},
      author={Kataoka, Misaki},
      author={Miyakoshi, Rika},
       title={On 2-twist-spun spherical {M}ontesinos knots},
        date={2020},
        ISSN={0218-2165},
     journal={J. Knot Theory Ramifications},
      volume={29},
      number={14},
       pages={2050099, 7},
         url={https://doi.org/10.1142/S0218216520500996},
      review={\MR{4216043}},
}

\bib{Joyce1982quandle}{article}{
      author={Joyce, David},
       title={A classifying invariant of knots, the knot quandle},
        date={1982},
     journal={J. Pure Appl. Algebra},
      volume={23},
      number={1},
       pages={37\ndash 65},
}

\bib{Kataoka2020twist}{thesis}{
      author={Kataoka, Misaki},
       title={On twist-spun knots with finite fundamental quandles},
        type={Master Thesis (Japanese)},
        date={January 2020},
}

\bib{Litherland1985symmetries}{incollection}{
      author={Litherland, R.~A.},
       title={Symmetries of twist-spun knots},
        date={1985},
   booktitle={Knot theory and manifolds ({V}ancouver, {B}.{C}., 1983)},
      series={Lecture Notes in Math.},
      volume={1144},
   publisher={Springer, Berlin},
       pages={97\ndash 107},
         url={https://doi.org/10.1007/BFb0075013},
      review={\MR{823283}},
}

\bib{Livingston1985stably}{article}{
      author={Livingston, Charles},
       title={Stably irreducible surfaces in {$S^4$}},
        date={1985},
        ISSN={0030-8730,1945-5844},
     journal={Pacific J. Math.},
      volume={116},
      number={1},
       pages={77\ndash 84},
         url={http://projecteuclid.org/euclid.pjm/1102707249},
      review={\MR{769824}},
}

\bib{Matveev1982distributive}{article}{
      author={Matveev, Sergei~Vladimirovich},
       title={Distributive groupoids in knot theory},
        date={1982},
     journal={Mat. Sb.},
      volume={161},
      number={1},
       pages={78\ndash 88},
}

\bib{Meeks1980topology}{article}{
      author={Meeks, William~H., III},
      author={Yau, Shing~Tung},
       title={Topology of three-dimensional manifolds and the embedding
  problems in minimal surface theory},
        date={1980},
        ISSN={0003-486X},
     journal={Ann. of Math. (2)},
      volume={112},
      number={3},
       pages={441\ndash 484},
         url={https://doi.org/10.2307/1971088},
      review={\MR{595203}},
}

\bib{Milnor1975Brieskorn}{incollection}{
      author={Milnor, John},
       title={On the {$3$}-dimensional {B}rieskorn manifolds {$M(p,q,r)$}},
        date={1975},
   booktitle={Knots, groups, and {$3$}-manifolds ({P}apers dedicated to the
  memory of {R}. {H}. {F}ox)},
      series={Ann. of Math. Studies, No. 84},
   publisher={Princeton Univ. Press, Princeton, N.J.},
       pages={175\ndash 225},
      review={\MR{0418127}},
}

\bib{Miyakoshi2022twist}{thesis}{
      author={Miyakoshi, Rika},
       title={On twist-spun knots with finite fundamental quandles},
        type={Master Thesis (Japanese)},
        date={January 2022},
}

\bib{Neumann1978Seifert}{inproceedings}{
      author={Neumann, Walter~D.},
      author={Raymond, Frank},
       title={Seifert manifolds, plumbing, {$\mu $}-invariant and orientation
  reversing maps},
        date={1978},
   booktitle={Algebraic and geometric topology ({P}roc. {S}ympos., {U}niv.
  {C}alifornia, {S}anta {B}arbara, {C}alif., 1977)},
      series={Lecture Notes in Math.},
      volume={664},
   publisher={Springer, Berlin},
       pages={163\ndash 196},
      review={\MR{518415}},
}

\bib{Papakyriakopoulos1957Dehn}{article}{
      author={Papakyriakopoulos, C.~D.},
       title={On {D}ehn's lemma and the asphericity of knots},
        date={1957},
     journal={Ann. of Math.},
      volume={66},
       pages={1\ndash 26},
}

\bib{Plotnick1983homotopy}{article}{
      author={Plotnick, Steven~P.},
       title={The homotopy type of four-dimensional knot complements},
        date={1983},
        ISSN={0025-5874},
     journal={Math. Z.},
      volume={183},
      number={4},
       pages={447\ndash 471},
         url={https://doi.org/10.1007/BF01173923},
      review={\MR{710763}},
}

\bib{Roseman1975spun}{article}{
      author={Roseman, D.},
       title={The spun square knot is the spun granny knot},
        date={1975},
     journal={Bol. Soc. Mat. Mexicana (2)},
      volume={20},
      number={2},
       pages={49\ndash 55},
      review={\MR{515725}},
}

\bib{Sakuma1982regular}{article}{
      author={Sakuma, Makoto},
       title={On regular coverings of links},
        date={1982},
        ISSN={0025-5831},
     journal={Math. Ann.},
      volume={260},
      number={3},
       pages={303\ndash 315},
         url={https://doi.org/10.1007/BF01461466},
      review={\MR{669298}},
}

\bib{Satoh2002surface}{incollection}{
      author={Satoh, Shin},
       title={Surface diagrams of twist-spun 2-knots},
        date={2002},
      volume={11},
       pages={413\ndash 430},
         url={https://doi.org/10.1142/S0218216502001718},
        note={Knots 2000 Korea, Vol. 1 (Yongpyong)},
      review={\MR{1905695}},
}

\bib{Scott1983geometries}{article}{
      author={Scott, Peter},
       title={The geometries of {$3$}-manifolds},
        date={1983},
        ISSN={0024-6093},
     journal={Bull. London Math. Soc.},
      volume={15},
      number={5},
       pages={401\ndash 487},
         url={https://doi.org/10.1112/blms/15.5.401},
      review={\MR{705527}},
}

\bib{Waldhausen1968irreducible}{article}{
      author={Waldhausen, Friedhelm},
       title={On irreducible {$3$}-manifolds which are sufficiently large},
        date={1968},
        ISSN={0003-486X},
     journal={Ann. of Math. (2)},
      volume={87},
       pages={56\ndash 88},
         url={https://doi.org/10.2307/1970594},
      review={\MR{224099}},
}

\bib{Zeeman1965twisting}{article}{
      author={Zeeman, E.~C.},
       title={Twisting spun knots},
        date={1965},
        ISSN={0002-9947},
     journal={Trans. Amer. Math. Soc.},
      volume={115},
       pages={471\ndash 495},
         url={https://doi.org/10.2307/1994281},
}

\end{biblist}
\end{bibdiv}
\end{document}